\documentclass{amsart}
\usepackage{amsxtra}
\usepackage[all]{xy}
\usepackage{palatino}
\usepackage{commath}
\usepackage{mathrsfs}
\usepackage{comment}

\usepackage{tikz-cd}
\usepackage[bookmarks=false]{hyperref}

\usepackage{amssymb,amsmath,amsthm,epsf,epsfig,dsfont,bbm}
 
\usepackage{upref, eucal}

\newcommand{\inj}{\hookrightarrow}

\numberwithin{equation}{section}
\def \mb{\mbox}

\usepackage{amssymb,amsmath,amsthm,epsf,epsfig,dsfont,bbm}

\usepackage{latexsym}
\usepackage{mathtools}
\usepackage{xcolor}
\newcommand{\nc}{\newcommand}

\newcommand{\ner}{\mb{\tiny Ner}}
\DeclareMathOperator{\Spec}{\mathrm{Spec}}

\newcommand{\Hom}{\mathrm{Hom}}
\newcommand{\End}{\mathrm{End}}
\newcommand{\Ext}{\mathrm{Ext}}

\newcommand{\Res}{\mb{Res}}

\newcommand{\beqar}{\begin{eqnarray*}}
\newcommand{\eeqar}{\end{eqnarray*}}

\newcommand{\oldmarginpar}[1]{}

\newcommand{\mcal}[1]{\mathcal{#1}}

\begin{document}
\nc{\Mhf}{{\bf M}_{0,H'}'}
\newcommand{\Muf}{{\bf M}'_{bal.U_1(p),H'}}
\newcommand{\Mho}{M_{0,H}}
\newcommand{\Muo}{M_{bal.U_1(p),H}}
\newcommand{\Mh}{M'_{0,H'}}
\newcommand{\Mu}{M'_{bal.U_1(p),H'}}
\newcommand{\Mucan}{{M'}_{bal.U_1(p),H'}^{{can}}}
\newcommand{\oMh}{\overline{M'}_{0,H'}}
\newcommand{\oMu}{\overline{M'}_{bal.U_1(p),H'}}
\newcommand{\pomega}{\omega^+}
\newcommand{\Eg}{{E'_1}\mid_S}
\newcommand{\Ig}{\mb{Ig}}

\newcommand{\Fab}{F_{\tiny \mb{abs}}}
\newcommand{\Frel}{F_{\tiny \mb{rel}}}
\nc{\sym}{\Psi}

\nc{\Gm}{\mathbb{G}_m}

\newcommand{\MH}{{\widehat{M}}'_{0,H'}}
\newcommand{\MU}{{\widehat{M}}'_{bal.U(p),H'}}
\newcommand{\Xfo}{{\widehat{X}}_!}
\newcommand{\Xfor}{{\widehat{X}}_{!!}}
\newcommand{\Xp}{\overline{X}_!}
\newcommand{\Xpp}{\overline{X}_{!!}}
\newcommand{\opi}{\overline{\pi}}
\newcommand{\ohat}[1]{\widehat{#1}}
\newcommand{\fsharp}{{\bold f}^\sharp_{{\tiny \Pi}}}
\newcommand{\mfrak}[1]{\mathfrak{#1}}
\newcommand{\bA}{{\bf A}}
\newcommand{\Mig}{\overline{M'}_{Ig,H'}}
\newcommand{\gp}{\mathfrak{p}}
\newcommand{\gP}{\mathfrak{P}}
\newcommand{\gl}{\mathfrak{l}}
\newcommand{\sA}{\mathcal{A}}

\newcommand{\bF}{\mathbb{F}}

\newcommand{\bb}[1]{\mathbb{#1}}
\newcommand{\cX}{\mathcal{X}}
\newcommand\A{\mathbb{A}}
\newcommand\C{\mathbb{C}}
\newcommand\G{\mathbb{G}}
\newcommand\N{\mathbb{N}}
\newcommand\T{\mathbb{T}}
\newcommand\sE{{\mathcal{E}}}
\newcommand\tE{{\mathbb{E}}}
\newcommand\sF{{\mathcal{F}}}
\newcommand\sG{{\mathcal{G}}}
\newcommand\GL{{\mathrm{GL}}}
\newcommand\bM{\mathbb{M}}
\newcommand\HH{{\mathrm H}}
\newcommand\mM{{\mathrm M}}
\newcommand\J{\mathfrak{J}}
\newcommand\sT{\mathcal{T}}
\newcommand\Qbar{{\bar{\Q}}}
\newcommand\sQ{{\mathcal{Q}}}
\newcommand\sP{{\mathbb{P}}}
\newcommand{\Q}{\mathbb{Q}}
\newcommand{\tH}{\mathbb{H}}
\newcommand{\Z}{\mathbb{Z}}
\newcommand{\R}{\mathbb{R}}
\newcommand{\F}{\mathbb{F}}
\newcommand\fP{\mathfrak{P}}
\newcommand\Gal{{\mathrm {Gal}}}
\newcommand{\Ou}{\mathcal{O}}
\newcommand{\legendre}[2] {\left(\frac{#1}{#2}\right)}
\newcommand\iso{{\> \simeq \>}}
\newcommand{\M}{\mathbb{M}}
\newcommand{\m}{\mathcal{m}}

\newtheorem {theorem} {\bf{Theorem}}[section]
\newtheorem {lemma}[theorem] {\bf Lemma}
\newtheorem {prop}[theorem]{\bf Proposition}
\newtheorem {proposition}[theorem] {\bf Proposition}
\newtheorem {exercise}{Exercise}[section]
\newtheorem {example}[theorem]{Example}
\newtheorem {remark}[theorem]{Remark}
\newtheorem {definition}[theorem]{\bf Definition}
\newtheorem {corollary}[theorem] {\bf Corollary}

\newtheorem{thm}{Theorem}
\newtheorem{cor}[thm]{Corollary}
\newtheorem{conj}[thm]{Conjecture}
\newtheorem{quen}[thm]{Question}
\theoremstyle{definition}
\newtheorem{claim}[thm]{Claim}
\newtheorem{que}[thm]{Question}
\newtheorem{lem}[thm]{Lemma}
\def \vp{\varphi}
\newcommand{\wh}[1]{\widehat{#1}}
\newcommand{\pip}{\Pi}
\newcommand{\Xdag}{\widehat{X}_\dagger}

\theoremstyle{definition}
\newtheorem{dfn}{Definition}

\theoremstyle{remark}


\newcommand {\stk} {\stackrel}
\newcommand{\map}{\rightarrow}
\def \bX{{\bf X}}
\newcommand {\orho}{\overline{\rho}}

\newcommand{\hGm}{\wh{\G}_m}

\newcommand{\Ner}[1]{{#1}^{\ner}}
\newcommand{\red}[1]{{\color{red} #1}}
\newcommand{\blue}[1]{{\color{blue} #1}}

\theoremstyle{remark}
\newtheorem*{fact}{Fact}
\makeatletter
\def\imod#1{\allowbreak\mkern10mu({\operator@font mod}\,\,#1)}
\makeatother

\title{Towards a mod-$p$ Lubin-Tate theory for $\GL_2$ over totally real fields}
\author{Debargha Banerjee}
\author{Vivek Rai}
\address{INDIAN INSTITUTE OF SCIENCE EDUCATION AND RESEARCH, PUNE, INDIA}

\begin{abstract} We show that the conjectural mod $p$ local Langlands correspondence can be realised in the mod $p$ cohomology of the Lubin-Tate towers. The proof utilizes a well known conjecture of Buzzard-Diamond-Jarvis \cite[Conj. 4.9]{BDJ10}, a study of completed cohomology of the ordinary and supersingular locus of the Shimura curves for a totally real field $F$ and of mod $l(\neq p)$ local Langlands correspondence as given by Emerton-Helm \cite{EmertonHelm14}. 
In the case of modular curves a similar theorem was obtained by Chojecki \cite{Cho15}. 
\end{abstract}

\subjclass[2010]{Primary: 11F80, Secondary: 11F11, 11F20, 11F30}
\keywords{Galois representations, Completed cohomology, Shimura curves}
\maketitle

\section{Introduction}
\label{Introduction}
The aim of this article is to study the relationship of the mod $p$ local Langlands correspondence with the mod $p$ completed cohomology of Shimura curves over a totally real field $F$ and with the mod $p$ cohomology of Lubin-Tate tower. Assuming a well known conjecture of Buzzard-Diamond-Jarvis \cite[Conj. 4.9]{BDJ10}, we prove in Theorem \ref{Mainthm} that the conjectural mod $p$ local Langlands correspondence lies in the mod $p$ completed cohomology of the supersingular part of the Shimura curve. 
By further assuming a Serre type Conjecture \ref{Serretype}, in Theorem \ref{LT}, we show that the conjectural mod $p$ local Langlands correspondence lies in the mod $p$ cohomology of the Lubin-Tate towers. 

In recent years quite a progress have been made to establish the local Langlands correspondence (both $p$-adic and mod $p$) but unfortunately the complete picture is not clear even for $\GL_2$. 
The mod $p$ and $p$-adic local Langlands correspondence has been established for $\GL_2(\Q_p)$ by Breuil, Berger, Colmez, Dospinescu, Paskunas and others. But for $\GL_2(E)$, when $E (\neq \Q_p)$ is a finite extension of $\Q_p$ the correspondence is not known. Roughly speaking the aim of the mod $p$ local Langlands is to give a correspondence between the $n$-dimensional continuous mod $p$ representation of $G_E:= \Gal(\bar{E}/E)$ and smooth admissible mod $p$ representation of $\GL_n(E)$. The picture on the Galois side is much simpler and it is the automorphic side which should be understood better in order to give such a correspondence. For $\GL_2(\Q_p)$ the classification of irreducible smooth admissible representation was started by Barthel-Livn\'e \cite{BL94} and completed by Breuil \cite{Breuil03} wherein he classified all the supersingular representation for $\GL_2(\Q_p)$ but for $\GL_2(E)$ the classification is still incomplete mainly because of the mysterious supersingular representations.

 The following are some basic differences between the smooth representations of $\GL_2(\Q_p)$ and $\GL_2(E)$:
\begin{itemize}
\item Unlike $\GL_2(\Q_p)$ the complete classification of supersingular representations for $\GL_2(E)$ is not known. The study of such representations were carried out by Breuil-Paskunas \cite{BP12} wherein they constructed infinite family of smooth admissible supersingular representations using the theory of diagrams. 
\item Unlike $\GL_2(\Q_p)$ there exist smooth representations of $\GL_2(E)$ which are not admissible \cite{Le19}, \cite{GS20}.
\item Under the assumption of Buzzard-Diamond-Jarvis conjecture it was shown by Hu \cite{Hu17} that for $\GL_2(E)$ there are non-trivial extensions between supersingular and principal series representations which is not the case for $\GL_2(\Q_p)$ (we will come back to this point later in the introduction).
\end{itemize}
Because of these reasons among others its quite difficult to establish mod $p$ local Langlands correspondence for $\GL_2(E)$. However there has some been recent progress on the Gelfand-Kirillov dimension of the possible candidate for such a correspondence, see \cite{HuWang20},\cite{breuil2020gelfandkirillov}.

Now coming to the global picture, Galois representations of $G_{\Q}:=\Gal(\bar{\Q}/\Q)$ coming from geometry (for instance modular Galois representation) are often realised in the \'etale cohomology of modular curves. In \cite[Theorem 1.2.6]{Emerton11}, Emerton showed that the mod~$p$ completed cohomology (constructed by taking direct limit over \'etale cohomology of modular curves for finite levels) of modular curve realises the mod $l$ local Langlands correspondence for all primes $l$, in particular it realises mod $p$ local Langlands correspondence (see \cite[Section 7]{Shin} for a different proof of this local-global compatibility result). This result is a special case of Buzzard-Diamond-Jarvis conjecture \cite[Conj. 4.9]{BDJ10} which states that the completed cohomology of Shimura curve (over a totally real field $F$) realises the mod $l$ local Langlands correspondences for $\GL_2(F_l)$ where $F_l$ denotes the completion of $F$ at the prime $l$ (for the precise statement see Conjecture~\ref{BDJ}). The local factors appearing in Conjecture \ref{BDJ} is studied in \cite{EmertonHelm14}, \cite{Helm13} for primes $l\nmid p$ and in \cite{Hu17},\cite{breuil2020gelfandkirillov},\cite{HuWang20} for primes lying above $p$.
Therefore to salvage information about mod $p$ local Langlands correspondence its natural to study the completed cohomology of Shimura curves. In fact in this article we will be also be concerned with the completed cohomology of ordinary and supersingular locus of Shimura curves. 

On the other hand the classical local Langlands correspondence for $\GL_n$ by Harris-Taylor \cite{HarrisTaylor01} was realised in the $l$-adic cohomology of Lubin-Tate towers. Also Carayol in \cite{Carayol1986}, \cite{Carayol88} using cohomology of vanishing cycles define `local fundamental representation' and showed that it decomposes into a product of Jacquet-Langlands correspondence and local Langlands correspondence.
Hence one hopes that mod $p$ local Langlands correspondence can also be realised inside the cohomology of Lubin-Tate towers (for details on Lubin-Tate tower see \cite{Dat12}). In fact for the case $l =p$, Chojecki \cite[Section 6]{Cho15} defines an analogue of the `local fundamental representation' and showed that the mod $p$ local Langlands correspondence for $\GL_2(\Q_p)$ injects into the cohomology of Lubin-Tate tower. Following Chojecki we will define `local fundamental representation' and show that it realises the conjectural mod $p$ local Langlands correspondence for Shimura curves. For a study of the $p$-adic cohomology of the Lubin-Tate towers see \cite{Scholze}.

To state our results we first define few notation. Fix an odd prime $p$, let $F$ be a totally real number field of finite degree $d>1$. Let $D$ be a quaternion algebra over $F$ such that $D$ splits at one real place of $F$ and non-split at other real places, also assume that there is only one prime $\mfrak{p}$ of $F$ lying above $p$ and $D$ splits at $\mfrak{p}$. One can associate to $D$ a system of quaternion Shimura curves $M_K$ indexed by compact open subgroups $K$ of $D_f = (D \otimes_F \mathbb{A}_f)^\times$ where $\mathbb{A}_f$ is finite adele of $F$. Let $\rho$ be an absolutely irreducible continuous two-dimensional representation of $G_F :=\Gal(\bar{F}/F)$ over $\mathbb{\bar{F}}_p$ and let $H^1_{\mathbb{\bar{F}}_p}$ denotes the completed cohomology of Shimura curves $M_K$. Let $F_{\mfrak{p}}$ be the completion of $F$ at $\mfrak{p}$, also assume that $\rho_{\mfrak{p}}:=\rho|_{\Gal(\bar{\Q}_p/F_{\mfrak{p}})}$ is absolutely irreducible. Then the conjectural mod $p$ local Langlands correspondence associates $\rho_{\mfrak{p}}$ to a supersingular representation of $\GL_2(F_{\mfrak{p}})$ say $\pi_{\mfrak{p}}$.  
We prove the following theorem.

\begin{theorem}\label{Mainthm}
Let $\rho : G_F \rightarrow \GL_2(\bar{\mathbb{F}}_p)$ be continuous, irreducible and totally odd representation such that $\rho_\mfrak{p}$ is irreducible. Assuming Conjecture \ref{BDJ} we have a continuous $G_F\times\GL_2(F_\mfrak{p})\times \Pi_{w\in\Sigma\setminus \Sigma'}(D\otimes_F F_w)^\times$- equivariant injection 
$$\rho\otimes \pi_\mfrak{p}\otimes (\otimes_{w\in \Sigma\setminus \Sigma'}\pi_w) \hookrightarrow H^1_{\bar{\F}_p, ss}[\mfrak{m}_\rho],$$
where $H^1_{\bar{\F}_p,\text{ss}}$ denotes the mod $p$ cohomology of the supresingular locus of $M_K$ (see \eqref{supersingular cohomology}).
\end{theorem}

Note that for $F= \Q$, the above theorem is proved by Chojecki \cite[Theorem 6.3]{Cho15}. Before we give the strategy of our proof let us mention one crucial difference between the proof of Theorem \ref{Mainthm} and of \cite[Theorem 6.3]{Cho15}. Chojecki uses the localisation functor at a supersingular representation $\pi$ of $\GL_2(\Q_p)$ and it is a result of Paskunas \cite{Paskunas13} that this functor decomposes the category of smooth admissible representation of $\GL_2(\Q_p)$ into a direct sum of categories of smooth representations of $\GL_2(\Q_p)$ such that the first direct summand has all those representations whose irreducible subquotients are isomorphic to $\pi$ and the second summand consists of smooth representations of $\GL_2(\Q_p)$ such that none of its irreducible subquotients are isomorphic to $\pi$. Such a decomposition does not occur for the category of smooth admissible mod $p$ representations of $\GL_2(E)$. In fact it is shown by Hu \cite{Hu17} that in tame case under the assumption of \cite[Conjecture 4.9]{BDJ10}, there exist non-trivial extensions between the principal series representations and supersingular representations. 

We state Conjecture \ref{Serretype} in Section \ref{lubin-tate} which is a generalisation of \cite[TH\'EOR\`EME, p.282]{Serre96}, comparing the system of eigenvalues for the Hecke operators in the Hecke algebra for two related quaternion algebras for $D$ and $\bar{D}$. Assuming Conjecture \ref{BDJ} and Conjecture \ref{Serretype}, we prove the following theorem.
\begin{theorem} \label{LT}
Let $\rho$ be as in Theorem \ref{Mainthm}, then we have a continuous $G_F\times\GL_2(F_\mfrak{p})\times \Pi_{w\in\Sigma\setminus \Sigma'}(D\otimes_F F_w)^\times$-equivariant injection
$$\rho\otimes \pi_\mfrak{p}\otimes (\otimes_{w\in \Sigma\setminus \Sigma'}\pi_w) \hookrightarrow   \hat{H}^1_{\text{LT}, \bar{\F}_p},$$
where $\hat{H}^1_{\text{LT}, \bar{\F}_p}$ is the first Lubin-Tate cohomology (see \ref{def lt}).
\end{theorem}
 \subsection*{Strategy of the proofs} The strategy for the proofs of above two theorems are summarised below:
 \begin{enumerate}
 \item For the proof of Theorem \ref{Mainthm} we first study the ordinary and supersingular part of Shimura curves and show that the cohomology of the ordinary part can be induced from Borel subgroup of  $\GL_2(F_\mfrak{p})$. This would be crucial because it will help us to show that completed cohomology of ordinary part of Shimura curve does not contain any supersingular representations since by definition the supersingular representations are not subquotient of the representations induced from Borel. We follow this by a thorough study of the mod $l$ local Langlands correspondence as given by Emerton-Helm \cite{EmertonHelm14} and Helm \cite{Helm13}. We prove a crucial Theorem \ref{vectors} which shows the existence of `new vectors' for such a correspondence. Using this we construct a tame level such that invariance under this tame level along with Conjecture \ref{BDJ} gives us an injection of the conjectural mod $p$ local Langlands correspondence in the completed cohomology of the supersingular part of the Shimura curve.
 \item For the proof of Theorem \ref{LT} we define the cohomology of the Lubin-Tate towers following Chojecki \cite[Section 6]{Cho15} and then we show that the cohomology of the supersingular part of Shimura curve injects into the cohomology of the Lubin-Tate towers. Using the above step we get an injection of the of conjectural mod $p$ local Langlands correspondence in the first cohomology of Lubin-Tate towers thereby completing the proof.
 \end{enumerate}






 \section{Acknowledgements}
The first author was partially supported by the SERB grant 
 and MTR/2017/000357. The second author thanks IISER Pune for the financial support provided during the preparation of the article. Both authors were partially supported by the SERB grant  CRG/2020/000223.

\section{Shimura curves}\label{shimura curves}
The aim of this section is to define Shimura curves with its ordinary and supersingular part with an emphasis on the decomposition of the ordinary part.
We follow Carayol \cite{Carayol86} in this section.
 Let $F$ be a totally real number field of finite degree $d>1$. Fix a finite prime $\mathfrak{p}$ of $F$ lying over a rational prime $p$ of $\mathbb{Q}$. Let $F_{\mfrak{p}}$ be the completion of $F$ at $\mfrak{p}$, $\mathcal{O}_{\mathfrak{p}}$ to its ring of integers and $k$ be its residue field with cardinality $q$. Let $D$ be the quaternion algebra over $F$ which splits at one real place $\tau$ and non-split at other real places.
 Further assume that $D$ is split at $\mathfrak{p}$, i.e. $D_{\mathfrak{p}}:= D\otimes_{F} F_{\mathfrak{p}} = M_2(F_{\mathfrak{p}})$ and let $G = {\rm{Res}}_{F/\Q}(D^\times)$.
  We are interested in Shimura curves corresponding to open compact subgroups of $G(\A_f)$, where $\mathbb{A}_f$ denotes the finite adeles of $F$. To separate out the $\mathfrak{p}$-part and prime to $\mathfrak{p}$-part (tame level) from an open compact subgroup $K$ of $G(\A_f)$, we define $K_{\mathfrak{p}}^0 = \GL_2(\mathcal{O}_\mathfrak{p})$ and $K_{\mathfrak{p}^n} = \lbrace g\in \GL_2(\mathcal{O}_{\mathfrak{p}}): g\equiv I_2 \text{ mod }\mathfrak{p}^n\rbrace$. 
 
Indexed by the compact open subgroups $K$ of $G(\mathbb{A}_f)$, we have a projective system of compact Riemann surfaces given by 
$$M_K(G)(\mathbb{C}):= G(\mathbb{Q})\setminus G(\mathbb{A}_f)\times (\mathbb{C}-\mathbb{R})/K.$$ 
Shimura has defined a canonical model $M_K(G)$ of $M_K(G)(\mathbb{C})$ over the field $F$. These $M_K(G)$ constitutes a projective system of complete smooth algebraic curves over $F$. 

Let $\Gamma$ be the restricted direct product of $(D\otimes F_v)^\times$ at all the finite places $v\neq  \mfrak{p}$. Let $K = K_{\mathfrak{p}^n}\times H$, for $H$ a compact open subgroup of $\Gamma$ and let $M_{n,H} := M_{K_{\mathfrak{p}^n}\times H}$. We are interested in the mod $p$ reduction of Shimura curves so it is necessary to work with its integral model.
 Carayol showed that for sufficiently small $H$, there exists an integral model on $\mathcal{O}_{(\mathfrak{p})}:= F\cap \mcal{O}_{\mfrak{p}}$ for $M_{0,H}$ which we again denote by $M_{0,H}$ and this model has a good reduction in $\mfrak{p}$. Carayol \cite{Carayol86} further constructs a scheme which we again denote by $M_{n,H}$ over $\mcal{O}_\mfrak{p}$ as ``the moduli space of full level $\mfrak{p}^n$ structures'' over $M_{0,H}$ and shows that it is an integral model of the Shimura curve $M_{n,H}$.
The $\mcal{O}_{\mfrak{p}}$-scheme $M_{n,H}$ lying above $M_{0,H}$ is regular, finite and flat \cite[p.200]{Carayol86}.
To define supersingular and ordinary part of $M_{n,H}$, we start with few more definitions and results about $M_{0,H}$. 

One knows that $M_{0,H}$ is not a solution to any moduli problem regarding abelian varieties, nevertheless Carayol defines \'etale $\mathcal{O}_{\mathfrak{p}}$-modules ${\bf{E}}_{1,H}$, ${\bf{E}}_{2,H}$, $\cdots$, ${\bf{E}}_{n,H}$ for each $H$. The objects ${\bf{E}}_{i,H}$ are finite flat group schemes of rank $q^{2i}$ and they play the part of the $\mathfrak{p}^n$-torsion points of the universal abelian variety had it been existed. Thus we can use these objects to define level structures at $\mfrak{p}$. Only finitely many ${\bf{E}}_{i,H}$ are defined for each $H$ but as $H$ becomes smaller and smaller, increasingly many of them are defined so that on the projective limit ${\pmb{M}}_0$ of the system $\lbrace{M}_{0,H}\rbrace$ we get a full divisible $\mcal{O}_{\mfrak{p}}$-module ${\bf{E}}_{\infty}$ (that is a $1$-dimensional $p$-divisible group with an action of $\mcal{O}_{\mfrak{p}}$). 

For every geometric point $x$ of the special fibre $ M_{0,H} \otimes ~k$, one can define a local divisible $\mathcal{O}_{\mathfrak{p}}$-module. One considers the \'etale covering $\pmb{M}_0 \to M_{0,H}$ and chooses a lift $y$ of $x$. The map then gives an isomorphism between the local ring at $x$ and local ring at $y$ and then we consider the pull back of the divisible $\mathcal{O}_{\mathfrak{p}}$-module $\bf{E}_\infty$ over ${\pmb{M}}_0$ via the morphism $y: \Spec\bar{k} \rightarrow {\pmb{M}}_0$. Any two choices of $y$ give rise to isomorphic divisible height two $\mcal{O}_{\mfrak{p}}$-module over $\bar{k}$, and we write ${\bf{E}}_{\infty}|_{x}$ for the result. By the classification given by Drinfeld in \cite{Drinfeld74} for the height $h$ divisible $\mcal{O}_{\mfrak{p}}$-module over $\bar{k}$, we have following two possibilities for ${\bf{E}}_{\infty}|_{x}$:
\begin{enumerate} 
\item ${\bf{E}}_{\infty}|_{x} = \Sigma_1 \times (F_{\mathfrak{p}}/\mathcal{O}_{\mathfrak{p}})$, we call $x$ to be ordinary in this case,
\item ${\bf{E}}_{\infty}|_{x} = \Sigma_2$, we call $x$ to be supersingular in this case,
\end{enumerate}
where $\Sigma_h$ is the unique (upto isomorphism) formal $\mathcal{O}_\mathfrak{p}$-module of height $h$.
Note that in the classical case (modular curves), one uses the same definition for ordinary and supersingular points on the curve \cite[\textsection V.3, Theorem 3.1, p144]{Silverman}. Carayol \cite{Carayol86} also proved that the set of supersingular points for $M_{0,H}$ is finite and non-empty.

By \cite[0.11]{Carayol86} denote $\mathscr{M}_{n,v(H)}$ (isomorphic to the spectrum of a finite abelian extension of $F$) to the finite $F$-scheme of the connected components of $M_{n,H}$ and it extends uniquely into a finite normal $\mcal{O}_{(\mfrak{p})}$-scheme $\pmb{\mathscr{M}}_{n,v(H)}$ (isomorphic to the spectrum of ring of $\mfrak{p}$-integers of the above extension) and the structure morphism  extends into a morphism $v: M_{n,H} \to \pmb{\mathscr{M}}_{n,v(H)}$ which is smooth outside supersingular points. We define the mod $p$ reduction of $M_{0,H}$ and $M_{n,H}$ as follows: Let $x: \Spec \bar{k} \to \pmb{\mathscr{M}}_{n,v(H)}$ be a point in the special fibre, then \cite[9.4.1]{Carayol86}
$$\bar{M}_{0,H} := M_{0,H} \times_{\pmb{\mathscr{M}}_{0,v(H)}} \Spec \bar{k}, \quad\quad\bar{M}_{n,H} := {M}_{n,H} \times_{\pmb{\mathscr{M}}_{n,v(H)}} \Spec \bar{k}.$$
Carayol showed that $\bar{k}$-scheme $\bar{M}_{n,H}$ is a proper, connected curve which is smooth outside the set of supersingular points. Let $\bar{M}_{n,H,\rm{ord}}$ be the curve obtained by removing the supersingular points from $\bar{M}_{n,H}$. So a $\bar{k}$-ordinary point of $\bar{M}_{n,H}$ is equivalent to 
\begin{itemize}
\item a $\bar{k}$-ordinary point of $\bar{M}_{0,H}$,
\item a surjective homomorphism $(\mfrak{p}^{-n}/\mcal{O}_{\mfrak{p}})^2 \to \mfrak{p}^{-n}/\mcal{O}_{\mfrak{p}} $ (the kernel is a direct factor of rank $1$ in the $\mcal{O}_{\mfrak{p}}/\mfrak{p}^n$-module $(\mfrak{p}^{-n}/\mcal{O}_{\mfrak{p}})^2$).
\end{itemize}

Let $a \in \rm{P}^1(\mcal{O}_\mfrak{p}/\mfrak{p}^n\mcal{O}_{\mfrak{p}})$ be a direct factor of rank $1$ in $\mcal{O}_{\mfrak{p}}/\mfrak{p}^n$-module $(\mfrak{p}^{-n}/\mcal{O}_{\mfrak{p}})^2$. Let $\bar{M}_{n,H,a}$ be the closed sub-scheme of $\bar{M}_{n,H}$ defined by the condition $\phi | a = 0$, where $\phi$ is the universal Drinfeld basis (\cite[Section 8.4]{Carayol86}) on $\bar{M}_{n,H}$. Define$\bar{M}_{n,H}^{\rm{ord}}$ to be the set of all ordinary points of $M_{n,H}$ and $\bar{M}_{n,H,a}^{\rm{ord}} := \bar{M}_{n,H, a} \cap \bar{M}_{n,H}^{\rm{ord}}$. Hence $\bar{M}_{n,H}^{\rm{ord}}$ is the disjoint union of $\bar{M}_{n,H,a}^{\rm{ord}}$, i.e.,
$$\bar{M}_{n,H}^{\rm{ord}} = \sqcup_{{a\in\rm{P}^1}(\mcal{O}_\mfrak{p}/\mfrak{p}^n\mcal{O})} \bar{M}_{n,H,a}^{\rm{ord}}.$$
The curves $\bar{M}_{n,H,a}$ is permuted by the action of $\GL_2(\mcal{O}_{\mfrak{p}})$ and hence it permutes $\bar{M}_{n,H,a}^{\rm{ord}}$, in fact the action of $\GL_2(\mcal{O}_{\mfrak{p}})$
factors through $\GL_2(\mcal{O}_{\mfrak{p}}/\mfrak{p}^n)$ (since $g \in \GL_2(\mcal{O}_{\mfrak{p}})$ takes $\bar{M}_{n,H,a}$ to $\bar{M}_{n,H,g\cdot a}$ for $a\in~\rm{P}^1(\mcal{O}_{\mfrak{p}}/\mfrak{p}^n) $).  

Let $M_{n,H}^{an}$ denote the analytification of $M_{n,H}$ which is a Berkovich space. Let $\pi : M_{n,H}^{an} \to \bar{M}_{n,H}$ be the reduction map. We define the ordinary and supersingular part of $M_{n,H}^{an}$ as 
$$M_{n,H}^{ \rm{ord}} = \pi^{-1}(\bar{M}_{n,H}^{\rm{ord}}),\quad\quad M_{n,H}^{\rm{ss}} = \pi^{-1}(\bar{M}_{n,H}^{\rm{ss}}),$$
where $\bar{M}_{n,H}^{\rm{ss}}$ is the complement of $\bar{M}_{n,H}^{\rm{ord}}$ in $\bar{M}_{n,H}^{an}$. Hence by the above discussion for sufficiently small $H$, we have a decomposition $M_{n,H}^{an} = M_{n,H}^{\rm{ord}}\sqcup M_{n,H}^{\rm{ss}}$.


\subsection{\'Etale sheaves and exact sequences}
Let $X$ be a scheme and $j : U \inj X$ be an open immersion. Let $Z = X \setminus U$ and $i : Z \to X$ be the inclusion map. Let $\mathcal{F}$ be an \'etale sheaf on $X$, we get a following exact sequence of sheaves on $X$
\begin{center}
$0\to j_!j^* \mathcal{F} \to \mathcal{F} \to i_*i^*\mathcal{F} \to 0,$
\end{center}
which gives the following long exact sequence of cohomology groups 
\begin{center}
$\cdots \to H^0(X, i_*i^*\mathcal{F}) \to H^1(X, j_!j^*\mathcal{F}) \to H^1(X, \mathcal{F}) \to H^1(X, i_*i^*\mathcal{F}) \to \cdots$
\end{center}
The cohomology with compact support is defined by $H^r_c(U,\mathcal{G}):= H^r(X, j_!\mathcal{G})$, for $r\geq 0$ and  any \'etale sheaf $\mathcal{G}$ on $U$. Also, by definition of $i^*$ we have $H^i(X, i_*i^*\mathcal{F}) = H^i(Z, i^*\mathcal{F})$ for $i =0$, $1$. Therefore we get 
\begin{center}
$\cdots \to H^0(Z, i^*\mathcal{F}) \to H^1_c(U, j^*\mathcal{F}) \to H^1(X, \mathcal{F}) \to H^1(Z, i^*\mathcal{F}) \to \cdots$
\end{center}
We will also consider the cohomology with support on $Z$. For this let $Z$ be a closed subvariety (or subscheme) of $X$. For any \'etale sheaf $\mathcal{F}$ on $X$ we have the following long exact sequence of cohomology groups 
\begin{center}
$\cdots \to H^r_Z(X, \mathcal{F}) \to H^r(X, \mathcal{F}) \to H^r(U, \mathcal{F}) \to H^{r+1}(X, \mathcal{F})\to \cdots $
\end{center}
By the general formalism of six operations for Berkovich spaces (see \cite{Berkovich93}) and by the comparison results of \'etale cohomology of schemes and its analytification (see \cite{Berkovich95}) we can use the above two long exact sequence (with compact support and with support on $Z$) in the case where $X = M_{n,H}^{an}$, $U = M_{n,H}^{\rm{ss}}$ and $Z = M_{n,H}^{ \rm{ord}}$, since $ M_{n,H}^{\rm{ss}}$ is an open analytic subspace of $M_{n,H}^{an}$ to get the following (later we will specialise the sheaf $\mcal{F}$ to $\bar{\F}_p$):
\begin{center}
\begin{align}
\cdots \to H^0(M_{n,H}^{ \rm{ord}}, i^*\mathcal{F}) \to H^1_c(M_{n,H}^{\rm{ss}}, j^*\mathcal{F}) \to H^1(M_{n,H}^{an}, \mathcal{F}) \to H^1(M_{n,H}^{ \rm{ord}}, i^*\mathcal{F}) \to \cdots, \label{compact support}\\
\cdots \to H^r_{M_{n,H}^{ \rm{ord}}}(M_{n,H}^{an}, \mathcal{F}) \to H^r(M_{n,H}^{an}, \mathcal{F}) \to H^r(M_{n,H}^{\rm{ss}}, \mathcal{F}) \to H^{r+1}_{M_{n,H}^{ \rm{ord}}}(M_{n,H}^{an}, \mathcal{F})\to \cdots \label{no compact support}
\end{align}
\end{center}

\subsection{Decomposition of the ordinary locus}
  The action of $g\in\GL_2(\mathcal{O}_\mathfrak{p}/\mathfrak{p}^n\mathcal{O})$ on $\bar{M}_{n,H}^{\rm{ord}}$ is such that it permutes the curves $\bar{M}_{n,H,a}^{\rm{ord}}$ for $a\in\mathbb{P}^1(\mathcal{O}_\mathfrak{p}/\mathfrak{p}^n\mathcal{O})$.
Let $b = \begin{pmatrix}
1\\
0
\end{pmatrix}
 \in  \mathbb{P}^1(\mathcal{O}_\mathfrak{p}/\mathfrak{p}^n\mathcal{O})$ and $B_{n}$ denote the Borel subgroup of upper triangular matrices in $\GL_2(\mathcal{O}_\mathfrak{p}/\mathfrak{p}^n\mathcal{O})$ which stabilizes $\bar{M}_{n,H,b}^{\rm{ord}}$. For $a\in \mathbb{P}^1(\mathcal{O}_\mathfrak{p}/\mathfrak{p}^n)$, define $M_{n,H,a}^{\rm{ord}} =\pi^{-1}(\bar{M}_{n,H,a}^{\rm{ord}})$.
 Hence $H^i(M_{n,H,b}^{\rm{ord}}, i^*F_{|M_{n,H,b}^{\rm{ord}}})$ is a representation of $B_n$. 
 
  By the above considerations we have for $i\geq 0$
 \begin{align}
 H^i(M_{n,H}^{\rm{ord}}, i^*F) &= \bigoplus_{a\in \mathbb{P}^1(\mathcal{O}_\mathfrak{p}/\mathfrak{p}^n\mathcal{O})}H^i(M_{n,H,{\rm{ord}}}^a, i^*F_{|M_{n,H,{\rm{ord}}}^a})\nonumber\\
 & \simeq \mathrm{Ind}_{B_{n}(\infty)}^{\GL_2(\mathcal{O}_\mathfrak{p}/\mathfrak{p}^n\mathcal{O})}H^i(M_{n,H,{\rm{ord}}}^b, i^*F_{|M_{n,H,{\rm{ord}}}^b})\label{decomposition1},
\end{align}
 and also 
 \begin{center}
\begin{align} 
 H^1_{M_{n,H}^{\rm{ord}}}(M_{n,H}, F) \simeq \mathrm{Ind}_{B_{n}(\infty)}^{\GL_2(\mathcal{O}_\mathfrak{p}/\mathfrak{p}^n\mathcal{O})}H^1_{M_{n,H,b}^{\rm{ ord}}}(M_{n,H}, F). \label{decomposition2}
 \end{align}
 \end{center}
 These results will be extremely helpful in the latter sections.

\section{mod $p$ completed cohomology of Shimura curves and admissibility results} \label{Completed cohomology}
In this section we define completed cohomology of Shimura curves which first appeared in \cite{Eme06}. In \cite{Emerton11}, Emerton showed that the completed cohomology of modular curves realizes both $p$-adic and mod $p$ local Langlands correspondence (\cite[Theorem 1.2.1, 1.2.6]{Emerton11}) and this is one of the reasons we are interested in completed cohomology.
The above mentioned results of Emerton in the mod $p$ setting (\cite[Theorem 1.2.6]{Emerton11}) is a special case of Buzzard-Diamond-Jarvis conjecture which we will state in the next section. We first define completed cohomology of Shimura curves. Note that we are implicitly using the comparison theorem for the \'etale cohomology of schemes and its analytification (for details see \cite{Berkovich95}) throughout the section.

Following the notation in Section \ref{shimura curves}, let  $M_{K_\mfrak{p}\times K^\mfrak{p}}$ be the Shimura curve $M_K(G)$ for the level $K= K_{\mfrak{p}}\times K^\mfrak{p}$. We assume that for a fixed rational prime $p$, there is only one prime $\mfrak{p}$ of $F$ lying above $p$.
 
 
 Define the completed cohomology for the Shimura curve $M_K$ as
$$H^m(M(K^\mfrak{p}), \bar{\F}_p):= \varinjlim\limits_{K_\mfrak{p}}H^m(M_{K_{\mfrak{p}}\times K^\mfrak{p}}^{an}, \bar{\F}_p),$$
it has a commuting smooth action of $G(\Q_p) = \GL_2(F_\mfrak{p})$ and a continuous action of $G_F$.

In Langlands correspondence ($p$-adic or mod $p$), one works with admissible representation on the automorphic side. It is desirable that the representation spaces we defined in Section \ref{Completed cohomology} are admissible. 
In this section we derive various results about the admissibility of the cohomology groups for Shimura curves. We begin by defining some basic notions from representation theory. Let $k$ be a field of characteristic $p$ and $G$ be a reductive group over $E$ (non-archimedean field of char $0$).
\begin{definition} A representation $V$ of $G$ over $k$ is said to be smooth at $x \in V$ if the Stab$(x) = \lbrace g\in G ~|~ g\cdot x =x\rbrace$ is open in $G$. The representation $V$ is said to be smooth if $V$ is smooth at all $x\in V$. Furthermore a smooth representation $V$ of $G$ is said to be admissible if for any compact set $K \subset G$, $V^K = \lbrace x\in V ~|~ g\cdot x = x ~\forall~ g\in K\rbrace$ ($K$-fixed vectors of $V$) is finite dimensional vector space over $k$. 
\end{definition}

The following proposition follows from \cite[Theorem 4.4.6]{Emerton10}.
\begin{prop} 
Let $V = {\rm{Ind}}_{P}^G W$ be a parabolic induction. If $V$ is admissible representation of $G$ over $k$, then $W$ is admissible representation of $P$ over $k$.
\end{prop}
We will apply the next lemma to show that a supersingular representation does not appear as a sub-quotient in the cohomology of the ordinary locus of the Shimura curves.
\begin{lemma} \label{unipotent}
For any admissible representation $(\pi, V)$ of the parabolic subgroup $P\subset G$ over $k$, the unipotent radical $U$ of $P$ acts trivially on $V$.
\end{lemma}
\begin{proof}
This is \cite[Lemma 3.3]{Cho15}.
\end{proof}
We derive the admissibility results for the cohomology groups defined in Section \ref{Completed cohomology}.

\begin{prop} \label{1}
The $\GL_2(F_\mfrak{p})$-representation $H^1(M(K^\mfrak{p}), \bar{\F}_p)$ is admissible.
\end{prop}
\begin{proof}
This is \cite[Theorem 2.1.5]{Eme06}.
\end{proof}

\begin{prop} The $\GL_2(F_\mfrak{p})$-representation $H^1_{M(K^\mfrak{p})^\emph{ord}}(M(K^\mfrak{p}), \bar{\F}_p) := \varinjlim_{K_\mfrak{p}} H^1_{M_{n,H}^\emph{ord}}(M^{an}(K_\mfrak{p}K^\mfrak{p}),\bar{\F}_p)$ is admissible.
\end{prop}
\begin{proof}
This is \cite[Proposition 3.9]{Cho15}.
\end{proof}

By \eqref{decomposition2}, the $\GL_2(F_\mfrak{p})$-representation $H^1_{M(K^\mfrak{p})^{\rm{ord}}}(M(K^\mfrak{p}), \bar{\F}_p)$ is isomorphic to the induced representation $$\mathrm{Ind}_{B_{\infty}(F_\mfrak{p})}^{\GL_2(F_\mfrak{p})}H^1_{M_{n,H,b}^{\rm{ ord}}}(M_{n,H}, \bar{\F}_p),$$
where $B_\infty(F_\mfrak{p})$ is the Borel subgroup of upper triangular matrices in $\GL_2(F_\mfrak{p})$. On this representation unipotent group acts trivially by Lemma \ref{unipotent} and hence it is induced from the tensor product of characters. This implies that $H^1_{M(K^\mfrak{p})^\emph{ord}}(M(K^\mfrak{p}), \bar{\F}_p)$ does not have any subquotient which is isomorphic to a supersingular representation. For future purpose we also define
\begin{align}
H^1(M(K^\mfrak{p})^{\text{ss}}, \bar{\F}_p):= \varinjlim_{K_{\mfrak{p}}}H^1(M_{n,K^\mfrak{p}}^{\text{ss}}, \bar{\F}_p),\quad H^1_{\bar{\F}_p, \text{ss}} :=\varinjlim_{K^\mfrak{p}}H^1(M(K^\mfrak{p})^{\text{ss}}, \bar{\F}_p). \label{supersingular cohomology}
\end{align}

\section{Buzzard-Diamond-Jarvis Conjecture and Modified mod $p$ LLC }\label{new vectors}
\subsection*{Buzzard-Diamond-Jarvis conjecture}
Let $D$ be a quaternion algebra over $F$ as in Section \ref{shimura curves}, i.e, D splits exactly at one infinite place. Let $U$ be an open compact subgroup of  $D_f^\times := (D\otimes \mathbb{A}_f)^\times$. In \cite{BDJ10}, the authors define $S^D(U) := H^1_{\acute{e}t}(Y_{U,\bar{F}}, \bar{\F}_p)$ ($Y_{U\bar{F}}$ is the Shimura curve corresponding to the compact open subgroup $U$) and $S^D := \varinjlim\limits_{U}S^D(U)$, where the limit is taken over all compact open subgroups of $D_f^\times$. 
With the notation introduced above $S^D =\varinjlim_{K^\mfrak{p}}H^1(M(K^\mfrak{p}), \bar{\F}_p)$. 
 Let $\rho: G_F \rightarrow \GL_2(\bar{\F}_p)$ be an irreducible continuous totally odd representation. Assume $\rho$ is modular and let $\Sigma$ be the finite set of finite places $w$ of $F$ such that $w$ divides $p$ or $D$ is ramified at $w$ or $\rho$ is ramified at $w$ or $\GL_2(\mathcal{O}_{F,w}) \nsubseteq U$, where $\mcal{O}_{F,w}$ denotes the completion of ring of integers $\mcal{O}_F$ of $F$ at the finite place $w$. For finite places $v$ of $F$ at which $D$ splits and $\GL_2(\mcal{O}_{F,v}) \subset U$ define the Hecke operators $T_v = \big[\GL_2(\mathcal{O}_{F,v})\begin{pmatrix} \varpi & 0\\
0&1
\end{pmatrix} \GL_2(\mathcal{O}_{F,v})\big]$ and $S_v = \big[ \GL_2(\mathcal{O}_{F,v})\begin{pmatrix} \varpi & 0\\
0&\varpi
\end{pmatrix} \GL_2(\mathcal{O}_{F,v})\big]$ which are elements of $\End_{\bar{\F}_p} (S^D(U))$ and where $\varphi$ denotes a uniformizer of $\mcal{O}_{F,v}$. Let ${\bf{T}}^\Sigma (U)$ denote the commutative  $\mathbb{\bar{F}}_p$-subalgebra of $\End_{\bar{\F}_p}(S^D(U))$ generated by $T_v$ and $S_v$ for all $v \not\in \Sigma$. Let $\mathfrak{m}^\Sigma_\rho$ be the maximal ideal of ${\bf{T}}^\Sigma (U)$ corresponding to $\rho$, i.e., generated by the operators $$T_v -S_v\text{tr}(\rho(\text{Frob}_v)) \quad\text{ and }\quad {\bf{N}}(v)-S_v \det(\rho(\text{Frob}_v))$$ 
for all $v\not\in\Sigma$.
Let $S^D(U)[\mathfrak{m}^\Sigma_\rho] = \lbrace f\in S^D(U) | \quad Tf = 0 \quad\forall \quad  T \in \mathfrak{m}^\Sigma_\rho\rbrace$. By Lemma 4.6 of \cite{BDJ10}, $S^D(U)[\mathfrak{m}^\Sigma_\rho]$ is independent of $\Sigma$, so we denote it by $S^D(U)[\mathfrak{m}_\rho]$. Also, by taking the limit over all compact open set $U$ of $D^\times$, $S^D[\mathfrak{m}_\rho] = \varinjlim\limits_{U}S^D(U)[\mathfrak{m}_\rho]$ becomes a representation of $D_f^\times$.
Now we state the Buzzard-Diamond-Jarvis conjecture, see \cite[Conjecture 4.9]{BDJ10}:

\begin{conj} \label{BDJ}Let $F$ be a totally real field and let 
$$\rho: G_F:=\Gal(\mathbb{\bar{Q}}/F)\rightarrow \GL_2(\bar{\mathbb{F}}_p)$$ be a continuous, irreducible and totally odd representation. Then the $\bar{\mathbb{F}}_p$ representation $S^D[\mathfrak{m}_\rho]$ of $G_F \times D_f^\times$ is isomorphic to a restricted tensor product $$ S^D[\mathfrak{m}_\rho]\cong  \rho \otimes ~(\otimes'_w \pi_w)$$ where $\pi_w$ is a smooth admissible representation of $D^\times_w$ such that 
\begin{itemize}
\item if $w$ does not divide $p$, then $\pi_w$ is the representation attached to $\rho_w := \rho|_{F_w}$ by the modulo $\ell$- local Langlands correspondence or Jacquet-Langlands correspondence (for the definition of $\pi_w$ in various cases see \cite[Section 4]{BDJ10}).
\item if $w$ divides $p$, then $\pi_w \neq 0$; moreover if both $F$ and $D$ are unramified at $w$ and $\sigma$ is any irreducible representation of $\GL_2(\mathcal{O}_{F_w})$, then $\Hom_{\GL_2(\mathcal{O}_{F_w})}(\sigma, \pi_w) \neq 0$ if and only if $\sigma \in W(\rho_w)$, where $W(\rho_w)$ is a certain set of Serre weights associated to $\rho_w$ (see \cite[Section 3]{BDJ10}).
\end{itemize}
\end{conj}

Let $p$ and $l$ be different odd primes. We recall the modified mod $p$ local Langlands correspondence for $\GL_2(E)$ given by Emerton and Helm in \cite[\textsection{5.2}]{EmertonHelm14} where
 $E$ is a finite extension of $\Q_l$. Let $\Ou_E$ be its ring of integers, $\gl$ be its maximal ideal, $k$ be its residue field and $q$ be the cardinality of $k$. Fix an algebraic closure $\bar{\F}_p$ of $\F_p$. Let $G_E := \Gal(\bar{\mathbb{Q}}_l / E)$ and $\rho : G_E \rightarrow \GL_2(\bar{\F}_p)$ be a continuous Galois representation. 
Let $\rho^{\rm{ss}}$ be the semisimplification of $\rho$ (so $\rho^{\rm{ss}}$ is either irreducible or reducible split). Let $\omega$ be the mod $p$ cyclotomic character of $G_E$.

We now write down  automorphic representation of $\GL_2(E)$ corresponding to $\rho$. Recall some facts about the extensions between mod $p$ representations of $\GL_2(E)$. Denote by $\pi_1$ cuspidal representation of $\GL_2(E)$ as defined by Vigneras \cite[Theoreme $3.c$]{Vigneras89}:
\begin{enumerate}
\item Up to isomorphism there is a unique non-split extension of $1$ by $\pi_1$.
\item Up to isomorphism there is a unique non-split extension of $\bar{\mid ~\mid}\circ \det$ by $\pi_1$.
\item Let $\rm{env}(\pi_1)$ be the unique extension of $1 \oplus (\bar{\mid ~\mid}\circ \det)$ by $\pi_1$ containing both the above non-split extensions as submodules.
\end{enumerate}

We recall the explicit description given my Emerton- Helm \cite[Proposition 5.2.1]{EmertonHelm14}. We denote the image of $\rho$ under the modified mod $p$ local Langlands correspondence by $\pi(\rho)$.

 We have the following possibilities:
\begin{itemize}
\item If $\rho^{\rm{ss}}$ is not a twist of $1\oplus \omega$, then $\pi(\rho)$ is uniquely determined by its supercuspidal support, and $\pi(\rho)$ is an irreducible representation of $\GL_2(E)$ hence generic (infinite dimension \cite[p. 657]{EmertonHelm14}).
\item If $\rho^{\rm{ss}}$ is a twist of $1\oplus \omega$, then we can assume that $\rho = 1 \oplus \omega$ (as LLC is compatible with twists). In the Banal case $(q \not\equiv \pm 1 ~\rm{ mod }~ p)$, we get
\begin{enumerate}
\item If $\rho$ is non-split, then $\pi(\rho) = \rm{St} \otimes (\bar{\mid~\mid} \circ \det)$.
\item If $\rho$ is split, then $\pi(\rho)$ is given by the unique non-split exact sequence 
$$0\rightarrow \rm{St}\otimes (\bar{\mid~\mid}\circ \det)\rightarrow \pi(\rho) \rightarrow (\bar{\mid~\mid}\circ \det)\rightarrow 0.$$
\end{enumerate}
\item Let $\rho^{\rm{ss}} = 1 \oplus \omega$ and $q\equiv -1$ mod $p$. We have following three possibilities.
\begin{enumerate}
\item if $\rho$ is split, then $\pi(\rho) = \rm{env}(\pi_1)$.
\item If $\rho$ is the non-split extension of $\omega$ by $1$, then $\pi(\rho)$ is given by the unique non-split extension 
$$0\rightarrow \pi_1 \rightarrow \pi(\rho) \rightarrow \bar{\mid ~\mid}\circ \det \rightarrow 0.$$
\item If $\rho$ is the non-split extension of $1$ by $\omega$, then $\pi(\rho)$ is given by the unique non split extension
$$0\rightarrow \pi_1 \rightarrow \pi(\rho) \rightarrow 1 \rightarrow 0.$$
\end{enumerate}
\item Let $\rho^{\rm{ss}} = 1 \oplus 1$ and $q \equiv 1$ mod $p$ (hence $\omega =1$), we have following two possibilities of $\pi(\rho)$.
\begin{enumerate}
\item If $\rho$ is split, then $\pi(\rho)$ is isomorphic to the universal extension of $1$ by $\rm{St}$ (and thus has length three).
\item If $\rho$ is non-split, then $\pi(\rho)$ corresponds to the non-split extension of $1$ by $\rm{St}$. 
\end{enumerate}
\end{itemize}

The following theorem is a generalisation of \cite[Theorem 5.2]{Cho15} to any local field $E $. Here we also include the case which was missing when $E = \Q_p$.
\begin{theorem}\label{vectors}
Let $\pi = \pi(\rho)$ be the mod $p$ admissible representation of $\GL_2(E)$ associated by the modified mod $p$ local Langlands correspondence given by Emerton-Helm \cite{EmertonHelm14} to a continuous Galois representation $\rho : G_E \rightarrow \GL_2(\bar{\mathbb{F}}_p)$. Then there exists an open compact subgroup $K$ of $\GL_2(\mathcal{O}_E) \subset \GL_2(E)$ such that $\dim_{\mathbb{\bar{F}}_p}\pi^K = 1$.
\end{theorem}

\begin{proof}
We do case by case analysis of the modified mod $\ell$ local Langlands correspondence.

 First assume that $\rho^{ss}$ is not a twist of $1\oplus \omega$, then by \cite[p. 491]{Helm13} $\pi(\rho)$ is an irreducible admissible representation of $\GL_2(E)$. The existence and uniqueness of Kirillov model for $\pi(\rho)$ is shown in \cite{Vigneras89}. It is known that the new vectors exists for such representations \cite[p. 64]{Vigneras89}. Therefore by the definition of new vectors there exists an open compact subgroup $K\subset \GL_2(\mathcal{O}_E)$ such that $\dim_{\bar{\mathbb{F}}_l}\pi(\rho)^K =1$.

Next we consider the Banal case (i.e. $q \not\equiv\pm 1$ mod $p$ and $\rho^{ss} = 1 \oplus \omega$), there are two possibilities for $\pi(\rho)$. Firstly when  $\pi(\rho) \simeq \rm{St} \otimes (\bar{\mid~\mid} \circ \det)$, we take $K = I$ where $I$ is the Iwahori subgroup of $\GL_2(\mcal{O}_E)$ as $\dim{\rm{St}}^I =1 $ and for any $g\in I$, $\bar{\mid~\mid} \circ \det$ is a trivial representation. Secondly when $\pi(\rho)$ is given by the following unique non-split exact sequence
$$0\rightarrow \rm{St}\otimes (\bar{\mid~\mid}\circ \det)\rightarrow \pi(\rho) \rightarrow (\bar{\mid~\mid}\circ \det)\rightarrow 0.$$
Taking $\GL_2(\mcal{O}_E)$-invariance for the above exact sequence we have 
$$0\rightarrow (\rm{St}\otimes (\bar{\mid~\mid}\circ \det))^{\GL_2(\mcal{O}_E)}\rightarrow \pi(\rho)^{\GL_2(\mcal{O}_E)} \rightarrow (\bar{\mid~\mid}\circ \det)^{\GL_2({\mcal{O}_E})}\rightarrow \Ext^1_{\GL_2(\mcal{O}_E)}(\bar{\mid~\mid}\circ \det, \rm{St}\otimes \bar{\mid~\mid}\circ \det)\to \cdots$$
We note that $\rm{St}^{\GL_2(\mcal{O}_E)} =0$, $\dim (\bar{\mid~\mid}\circ \det)^{\GL_2({\mcal{O}_E})} =1$ and $\Ext^1_{\GL_2(\mcal{O}_E)}(\bar{\mid~\mid}\circ \det, \rm{St}\otimes \bar{\mid~\mid}\circ \det) = \Ext^1_{\GL_2(\mcal{O}_E)}(1, \rm{St}) = \Ext^1_{\GL_2(\emph{k})}(1, \rm{St}) = 0$, therefore $K = \GL_2(\mcal{O}_E)$ works.

Next we consider the non-banal case with $q\equiv -1$ mod $p$. In this case we have three possibilities for $\pi(\rho)$. 
\begin{itemize}
\item Suppose $\pi(\rho)$ is given by $0 \to \pi_1\to\pi(\rho) \xrightarrow{\alpha}\ 1 \oplus (\bar{\mid~\mid}\circ \det) \to 0$, taking $\GL_2(\mcal{O}_E)$-invariance we get
$$0\to \pi(\rho)^{\GL_2(\mcal{O}_E)} \to 1\oplus (\bar{\mid~\mid}\circ \det)^{\GL_2(\mcal{O}_E)} \xrightarrow{\delta} \Ext_{\GL_2(\mcal{O}_E)}(1, \pi_1)\to \cdots,$$
since $\pi_1^{\GL_2(\mcal{O}_E)} = 0$. 
Let $\mcal{E}$ be the extension $0\to \pi_1 \to \alpha^{-1}(1,0) \to 1 \to 0 $. 
Let $f_{(a,b)} \in \Hom_{\GL_2(\mcal{O}_E)}(1, 1 \oplus (\bar{\mid~\mid}\circ \det)) = (1 \oplus (\bar{\mid~\mid}\circ \det) )^{\GL_2(\mcal{O}_E)}$ such that $1 \mapsto (a,b) \in 1 \oplus \bar{\mid~\mid}\circ \det$. Note that as $\GL_2(\mcal{O}_E)$-representation $\bar{\mid~\mid}\circ \det$ is just the trivial representation. We give the definition of the map $\delta$ as follows: 
 the connecting homomorphism $\delta$ evaluated at $f_{(a,b)}$ is given by the pull back of the following diagram:
\[
\begin{tikzcd}
  0 \arrow[r] & \pi_1 \arrow[d, "\rm{Id}"] \arrow[r] & \delta(f_{(a,b)}) \arrow[d] \arrow[r] & 1 \arrow[d, "f_{(a,b)}"] \arrow[r] & 0 \\
  0 \arrow[r] & \pi_1 \arrow[r] & \pi(\rho) \arrow[r, "\alpha"] & 1 \oplus \bar{\mid~\mid}\circ \det \ar[r] & 0,
\end{tikzcd}
\]
i.e., $\delta(f_{(a,b)}) = \pi(\rho)\times_{\bar{\F}_p} {\bar{\F}}_p = \lbrace (v,x) \in \pi(\rho)\times \bar{\F}_p~|~ \alpha(v) = f_{(a,b)}(x) \rbrace  = \lbrace (v,x) \in \pi(\rho)\times \bar{\F}_p~|~ \alpha(v) = xf_{(a,b)}(1)= x(a,b) \rbrace$. Its is easy to see that $\delta(f_{(a,b)}) = \alpha^{-1}(\bar{\F}_p(a,b))$. Therefore $\delta(f_{(a,b)}) = (a+b)\mcal{E}$, where $\mcal{E}$ is as above. By dimension comparison we get $\dim \pi(\rho)^{\GL_2(\mcal{O}_E)} =1$, because $\dim(\rm{Im}(\delta)) =1$ and $\dim(1\oplus  \bar{\mid~\mid}\circ \det)^{\GL_2(\mcal{O}_E)} = 2$.
\item Suppose $\pi(\rho)$ is given by $0 \to \pi_1 \to \pi(\rho) \to 1 \to 0$, taking $I_1$-invariance where $I_1 \subset \GL_2(\mcal{O}_E)$ is such that $I_1$ mod $\mfrak{l}$ is unipotent upper triangular subgroup of $\GL_2(k)$, we get
$$0\to \pi^{I_1} \to\pi(\rho)^{I_1}\to 1^{I_1} \to \Ext^1_{I_1}(1, \pi_1)\to\cdots$$
By Proposition 24 \cite{Vigneras89}, $\pi_1^{I_1} = 0$, also as $I_1$ is pro-$p$ group, any smooth $I_1$-representation is semi-simple. Therefore, $\Ext^1_{I_1}(1, \pi_1) = 0$ and $\dim(\pi(\rho))^{I_1} =1$ and we can take $K =I_1$.
\item Similarly, if $\pi(\rho)$ is given by $0\to\pi_1\to\pi(\rho) \to\bar{\mid ~\mid }\circ\det \to 0$, taking $I_1$-invariance and using the same reasoning as above we see that $K = I_1$ works in this case also. 
\end{itemize}
The last remaining non-banal case is when $q \equiv 1\text{ mod } p$ and in this case $\pi(\rho)$ can have two possibilities. 
\begin{itemize}
\item Suppose $\pi(\rho)$ is given by $0 \to \rm{St} \to \pi(\rho) \xrightarrow{\alpha} 1 \oplus 1\to 0$, taking $\GL_2(\mcal{O}_E)$-invariance we have 
$$0 \to \pi(\rho)^{\GL_2(\mcal{O}_E)} \to (1 \oplus 1)^{\GL_2(\mcal{O}_E)} \xrightarrow{\delta} \Ext^1_{\GL_2(\mcal{O}_E)}(1 \oplus 1, \rm{St})\to \cdots$$
Let $\mcal{E}$ be the $\GL_2(\mcal{O}_E)$-exact sequence $0 \to \rm{St}\to \alpha^{-1}(\bar{\F}_p, 0)\to 1\to 0$, and $f_{(a,b)} \in (1\oplus 1)^{\GL_2(\mcal{O}_E)} = \Hom_{\GL_2(\mcal{O}_E)}(1, 1\oplus 1)$ such that $1\mapsto (a,b)$. By definition of the connecting homomorphism $\delta$, $\delta(f_{(a,b)})$ is given by the pull back of the following diagram
\[
\begin{tikzcd}
  0 \arrow[r] & \rm{St} \arrow[d, "\rm{Id}"] \arrow[r] & V \arrow[d] \arrow[r] & 1 \arrow[d, "f_{(a,b)}"] \arrow[r] & 0 \\
  0 \arrow[r] & \rm{St} \arrow[r] & \pi(\rho) \arrow[r, "\alpha"] & 1 \oplus 1 \ar[r] & 0,
\end{tikzcd}
\]
i.e., $V = \lbrace (v,x) \in \pi(\rho)\times \bar{\F}_p |~ \alpha(v) = f_{(a,b)}(x) = x\cdot f_{(a,b)}(1) = x\cdot(a,b)\rbrace$, as before we identify $V$ with $(a+ b)\mcal{E}$. By comparing dimensions we have $\dim\pi(\rho)^{\GL_2(\mcal{O}_E)} = \dim\ker \delta =1$, hence we can take $K = \GL_2(\mcal{O}_E)$.
\item Suppose $\pi(\rho)$ is given by the unique non-split exact sequence $0 \to \rm{St}\to \pi(\rho)\to 1 \to 0$, here the arguments are similar to \cite[Theorem 5.2, (4)]{Cho15} so we do not prove it here. 
\end{itemize}

\end{proof}

\section{Proof of Theorem \ref{Mainthm}}\label{the proof}
\begin{proof}

Taking direct limit over all compact open subgroup $K_\mfrak{p} \subset \GL_2(F_{\mfrak{p}})$ in the exact sequence \eqref{no compact support} we get the following exact sequence
$$\cdots \xrightarrow{f_0} H^1_{M(K^\mfrak{p})^{ \rm{ord}}}(M(K^\mfrak{p})^{an}, \bar{\F}_p) \xrightarrow{f_1} H^1(M(K^\mfrak{p})^{an}, \bar{\F}_p) \xrightarrow{f_2} H^1(M(K^\mfrak{p})^{\rm{ss}}, \bar{\F}_p) \xrightarrow{f_3} H^2_{M(K^\mfrak{p})^{ \rm{ord}}}(M(K^\mfrak{p})^{an}, \F_p)\to \cdots $$
From this we get a following short exact sequence
$$0\to {\rm{Im}}~f_1 \to H^1(M(K^\mfrak{p})^{an}, \bar{\F}_p) \to {\rm{Im}}~ f_2 \to 0,$$
where ${\rm{Im}}~f_1\simeq H^1_{M(K^\mfrak{p})^{ \rm{ord}}}(M(K^\mfrak{p})^{an}, \bar{\F}_p) /\ker(f_1)$. Applying $\Hom_G(\pi_\mfrak{p}, \_)$ to the above sequence where $\pi_\mfrak{p}$ corresponds to $\rho|_\mfrak{p}$ under the conjectural mod $p$ LLC for $\GL_2(F_\mfrak{p})$, we get
$$0\to \Hom_G(\pi_\mfrak{p}, H^1_{M(K^\mfrak{p})^{ \rm{ord}}}(M(K^\mfrak{p})^{an}, \bar{\F}_p) /\ker(f_1)) \to \Hom_G(\pi_\mfrak{p},  H^1(M(K^\mfrak{p})^{an}, \bar{\F}_p)) \to \Hom_G(\pi_\mfrak{p}, {\rm{Im}}~f_2) \cdots$$
Since $\pi_\mfrak{p}$ is a supersingular representation (because $\rho|_\mfrak{p}$ is an irreducible representation) it cannot occur as a sub-quotient of $H^1_{M(K^\mfrak{p})^{ \rm{ord}}}(M(K^\mfrak{p})^{an}, \bar{\F}_p)$ because the cohomology is induced from Borel. Hence we get the following injection 
$$ \Hom_G(\pi_\mfrak{p},  H^1(M(K^\mfrak{p})^{an}, \bar{\F}_p)) \hookrightarrow \Hom_G(\pi_\mfrak{p}, {\rm{Im}}~f_2) \hookrightarrow\Hom_G(\pi_{\mfrak{p}}, \emph{H}^1(M(K^\mfrak{p})^{\rm{ss}}, \bar{\F}_p) ) ,$$
taking direct limit over all compact open subgroup $K^\mfrak{p}$ we get (note that $\Hom$ commutes with direct limit because $\pi_\mfrak{p}$ is a simple $\bar{\F}_p[G]$-module)
$$\Hom_G(\pi_\mfrak{p},  H^1_{\bar{\F}_p}) \hookrightarrow \varinjlim_{K^\mfrak{p}}\Hom_G(\pi_\mfrak{p}, {\rm{Im}}~f_2) \hookrightarrow\Hom_G(\pi_\mfrak{p}, \emph{H}^1_{\bar{\F}_p, {\rm{ss}}}) .$$
Taking $\mfrak{m}_\rho$-torsion where $\mfrak{m}_\rho$ is as in Section \ref{new vectors}, and using Conjecture \ref{BDJ} we get the following
$$\Hom_G(\pi_\mfrak{p},   H^1_{\bar{\F}_p}[\mfrak{m}_\rho]) \hookrightarrow \Hom_G(\pi_{\mfrak{p}}, \emph{H}^1_{\bar{\F}_p, {\rm{ss}}}[\mfrak{m}_\rho] ).$$
Tensoring both sides by $\pi_\mfrak{p}$ over $\bar{\F}_p$ and using Conjecture \ref{BDJ} we have
$$\pi_\mfrak{p} \otimes _{\bar{\F}_p} \Hom_G(\pi_\mfrak{p},   \rho \otimes(\otimes'_w\pi_{w})) \hookrightarrow \pi_\mfrak{p} \otimes _{\bar{\F}_p} \Hom_G(\pi_\mfrak{p}, \emph{H}^1_{\bar{\F}_p, {\rm{ss}}}[\mfrak{m}_\rho] ).$$
Now $\rho \otimes(\otimes'_{w\nmid p}\pi_w) \hookrightarrow  \Hom_G(\pi_\mfrak{p},   \rho \otimes(\otimes'_w\pi_{w}))$ (see \cite[p.269, Proposition 2(i)]{Bourbaki}) and using evaluation map $$\text{ev}: \pi_\mfrak{p} \otimes \Hom_G(\pi_\mfrak{p},\emph{H}^1_{\bar{\F}_p, {\rm{ss}}}[\mfrak{m}_\rho] ) \rightarrow \emph{H}^1_{\bar{\F}_p,{\rm{ ss}}}[\mfrak{m}_\rho],\quad(v, f) \mapsto f(v)$$
we get that $ \pi_\mfrak{p} \otimes _{\bar{\F}_p} \Hom_G(\pi_\mfrak{p}, \emph{H}^1_{\bar{\F}_p,{\rm{ss}}}[\mfrak{m}_\rho] ) \hookrightarrow \emph{H}^1_{\bar{\F}_p, {\rm{ss}}}[\mfrak{m}_\rho]  $ since $\pi_\mfrak{p}$ is irreducible representation of $G$ over $\bar{\F}_p$. Hence we have the following injection
\begin{align}
\rho\otimes\pi_\mfrak{p} \otimes (\otimes'_{w\nmid p}\pi_w) \hookrightarrow  \emph{H}^1_{\bar{\F}_p, {\rm{ss}}}[\mfrak{m}_\rho]. \label{inclusion}
\end{align}
Let $\Sigma'=\lbrace w\in \Sigma~| ~D \text{ splits at } w\rbrace$ and put $K = K_{\Sigma'} K^\Sigma$ where $K^\Sigma = \Pi_{w\not\in \Sigma}\GL_2(\mcal{O}_w)$ and $K_{\Sigma'}= \Pi_{w\in \Sigma'} K_w$ where we choose $K_w$ using Theorem \ref{vectors}, i.e., $\dim_{\bar{\F}_p} \pi_w^{K_w} = 1$. Taking $K$-invariance in \eqref{inclusion} we have
\begin{align}
\rho\otimes \pi_\mfrak{p}\otimes (\otimes_{w\in \Sigma\setminus \Sigma'}\pi_w) \hookrightarrow  \emph{H}^1_{\bar{\F}_p, {\rm{ss}}}[\mfrak{m}_\rho]^K, \label{mainthm}
\end{align}
as $G_F\times\GL_2(F_\mfrak{p})\times \Pi_{w\in\Sigma\setminus \Sigma'}(D\otimes_F F_w)$-representation.
If $\Sigma = \Sigma'$, then by \eqref{mainthm} we have $\rho \otimes \pi_\mfrak{p} \hookrightarrow  \emph{H}^1_{\bar{\F}_p, {\rm{ss}}}[\mfrak{m}_\rho]$ as $G_F\times\GL_2(F_\mfrak{p})$-representation which is similar to the result obtained by Chojecki~[Section 6.3]\cite{Cho15}. 
\end{proof}

\section{Lubin-Tate towers and the fundamental representation}\label{lubin-tate}
In this section we connect the cohomology of the Lubin-Tate tower with the completed cohomology of the Shimura curves. The motivation to look into the cohomology of Lubin-Tate towers comes from the classical local Langlands correspondence as it was realised in the cohomology of the Lubin-Tate tower. 
The aim of this section is two-fold, firstly to give a description of the supersingular locus of the Shimura curve and secondly to define the local fundamental representation. We will follow Carayol \cite{Carayol1986} for the both the definitions and notation.


Let $\mfrak{p}$ be a fixed uniformiser of $F_{\mfrak{p}}$, $\mcal{O}_{\mfrak{p}}$ be its ring of integers and $k$ be its residue field. Divisible $\mcal{O}_{\mfrak{p}}$-modules over $\bar{k}$ and their deformations appear naturally in the theory of bad reductions of Shimura curves. For an overview of the definitions and properties of divisible $\mcal{O}_{\mfrak{p}}$-modules we refer the reader to Drinfeld \cite{Drinfeld74} or Carayol \cite[Appendix]{Carayol86}. Here we state the results which is needed for our purpose. 
We start by discussing few basics about formal $\mcal{O}_{\mfrak{p}}$-modules and their deformations. 
For every integer $h\geq 1$, there exists a unique formal $\mcal{O}_{\mfrak{p}}$-module $\Sigma_h$ of height $h$ over $\bar{k}$ upto isomorphism. 
Let $\mcal{C}$ be the category of complete local noetherian $\hat{\mcal{O}}^{nr}_{\mfrak{p}}$ (maximal unramified)-algebras with residue field $\bar{k}$. Let $G$ be a divisible $\mcal{O}_{\mfrak{p}}$-module on $\bar{k}$ of height $h$. Recall that a deformation of $G$ over an object $A \in \mcal{C}$ is a pair $(\mcal{G}, \iota)$ consisting of formal $\mcal{O}_{\mfrak{p}}$-module $\mcal{G}$ over $A$ which is equipped with an isomorphism $\iota : G \to \mcal{G}_{\bar{k}}$ of formal $\mcal{O}_{\mfrak{p}}$-module over $\bar{k}$ where $\mcal{G}_{\bar{k}}$ denotes the reduction of $\mcal{G}$ modulo the maximal ideal $\mfrak{m}_A$ of $A$.
The functor which associates $A\in\mcal{C}$ to the set of deformations of $G$ over $A$ is representable by the ring $D_0^{G} \in \mcal{C}$ isomorphic to the ring of formal series in $h-1$ variables over $\hat{\mcal{O}}^{nr}_{\mfrak{p}}$.  We are interested in $h=2$.

Let $D$ and $\Gamma$ be as in Section \ref{shimura curves}. Let $\bar{D}$ be a quaternion algebra over $F$ which is non-split at $\tau$, $\mfrak{p}$ and also at all places where $D$ is non-split. Let $\bar{G} = \Res_{F/\Q}(\bar{D}^*)$, so $\bar{G}$ is a twisted inner form on $G = \Res_{F/\Q}(D^*)$. The center of $\bar{G}$ is identified with the center $Z$ of $G$. Let $\hat{Z(\Q)}$ denote the closure of $Z(\Q)$ in $Z(\A^f)$. Let us fix $\Sigma$ a formal  $\mcal{O}_{\mfrak{p}}$-module of height $2$ over $\bar{k}$ (upto isogeny) and isomorphisms $\bar{D}_{\mfrak{p}}^*\simeq$ Aut $\Sigma$, $\bar{D}\otimes \A^{f, \mfrak{p}} \simeq D \otimes \A^{f, \mfrak{p}}$. Denote $V_\mfrak{p}$ to be the vector space $F_\mfrak{p}^2$ with an action of $\GL_2(F_{\mfrak{p}})$ from right. Consider the product 
$$\tilde{\Delta} = \text{W}(\bar{k}/k) \times \text{Isom}(\wedge^2 V_\mfrak{p}, v(\Sigma)\otimes \rm{VL}_{\hat{\bar{\eta}}}),$$
where $L$ is the Lubin-Tate group over $\mcal{O}_{\mfrak{p}}$ and $\hat{\bar{\eta}} = \Spec (\hat{\bar{\F}}_\mfrak{p})$. For the definition of the functors $v$ and $V$ see \cite[p.437, Section 7.7]{Carayol1986}. The set $\tilde{\Delta}$ is a principal homogeneous space over the group ${\rm{W}}(\bar{k}/k)\times F_p^*$. Consider its subgroup isomorphic to $\Z$ consisting of pairs $(\sigma^n, \mfrak{p}^n)$ where $\sigma$ denotes the geometric Frobenius. Now consider the set $\Delta =\tilde{\Delta}/\Z$, i.e., we identify the elements $(w,\psi)\sim (w\sigma^n, \psi\mfrak{p}^n)$ for $(w,\psi) \in \tilde{\Delta}$. The set $\Delta$ now is a principal homogeneous space over $F_\mfrak{p}^*$ with an action of $\text{W}(F_\mfrak{p}^{ab}/F_\mfrak{p})$, $\GL_2(F_\mfrak{p})$ and Aut $\Sigma$, see \cite[p.445, Section 9.3]{Carayol1986}. Now the supersingular orbit $S = \varprojlim_{n,H} S_{n,H}$ admits a description as the $G(\A^f) \times W(F^{ab}_\mfrak{p}/F_{\mfrak{p}})$-equivariant quotient set
$$S = \Delta \times_{\bar{D}^\times(F)} \Gamma.$$
Every element $\delta\in\Delta$ corresponds to an equivalence class of polarised formal $\mcal{O}_\mfrak{p}$-module (see \cite[p.444 Section 9.2]{Carayol1986})
so for each $\delta\in \Delta$, let $\rm{LT}_{\delta}$ consider the Lubin-Tate tower which is the generic fibre of the deformation space of the formal group attached to $\delta$ and let $\rm{LT_\Delta} =\sqcup_{\delta\in \Delta} \text{LT}_{\delta}$ (see \cite{Dat12} for details on Lubin-Tate tower). Therefore pulling back the supersingular locus to characteristic zero we get 
$$M^{\rm{ss}} = \varprojlim_{n,H} M_{n,H}^{\rm{ss}} = \text{LT}_{\Delta} \times_{\bar{D}^\times(F)} \Gamma .$$
One can also get a description at a finite level as follows: let $K_n$ denote the kernel of $\bar{D}^{\times}(\mcal{O}_{\mfrak{p}})\to \bar{D}^\times({\mcal{O}_{\mfrak{p}}/\mfrak{p}^n\mcal{O}_{\mfrak{p}}})$ and let $H$ be an open compact subgroup of $\Gamma$. Then, $$M_{n,H}^{\rm{ss}} \simeq \text{LT}_{\Delta/K_n} \times_{{\bar{D}}^\times (F)} \Gamma/H,$$
where $\text{LT}_{\Delta/K_n} = \sqcup_{\delta \in\Delta/K_n} \text{LT}_\delta (\mfrak{p}^n)$, where $\text{LT}_{\delta}(\mfrak{p}^n)$ denotes the generic fibre of the deformation space of the formal group attached to $\delta$ with $\mfrak{p}^n$-level structure.

Following Chojecki \cite[Section 6.1]{Cho15}, we define the local fundamental representation as follows.
\begin{definition} \label{def lt}
With the notation as above define the local fundamental representation as 
$$\hat{H}^1_{\emph{LT}, \bar{\F}_p} = \varinjlim_n H^1(\emph{LT}_{\Delta/K_n}, \bar{\F}_p).$$ 
\end{definition}
By the description of the supersingular points at a finite level we have
\begin{align*}
H^1(M_{n,H}^\text{ss}, \bar{\F}_p) &= H^1(\text{LT}_{\Delta/K_n}\times_{\bar{D}^{\times}(F)} \Gamma/H, \bar{\F}_p)\\
&= \lbrace f:\bar{D}^\times (F) \setminus \Gamma/H \to H^1(\text{LT}_{\Delta/K_n}, \bar{\F}_p)\rbrace.
\end{align*}
First taking direct limit over $n$ and then over $H$ we get 
\begin{align*}
\varinjlim_nH^1(M_{n,H}^\text{ss}, \bar{\F}_p) &\simeq \lbrace f:\bar{D}^\times (F) \setminus \Gamma/H \to \varinjlim_n H^1(\text{LT}_{\Delta/K_n}, \bar{\F}_p)\rbrace,\\
\varinjlim_H\varinjlim_n H^1(M_{n,H}^\text{ss}, \bar{\F}_p) &\simeq \lbrace f: \bar{D}^\times (F)\setminus \Gamma\to \hat{H}^1_{\text{LT}, \bar{\F}_p} \rbrace,\\
 H^1_{\bar{\F}_p, \text{ss}} &\simeq \lbrace f: \bar{D}^\times (F)\setminus \bar{D}^\times(\A^f) \to \hat{H}^1_{\text{LT}, \bar{\F}_p} \rbrace^{\bar{D}^\times (F_{\mfrak{p}})},\\
 &\simeq ( \lbrace f: \bar{D}^\times (F)\setminus \bar{D}^\times(\A^f) \to \bar{\F}_p \rbrace \otimes_{\bar{\F}_p} \hat{H}^1_{\text{LT}, \bar{\F}_p})^{\bar{D}^\times (F_{\mfrak{p}})}.
\end{align*}
Denoting $\textbf{F} = \lbrace f: \bar{D}^\times (F)\setminus \bar{D}^\times(\A^f) \to \bar{\F}_p \rbrace$, then 
\begin{align}
H^1_{\bar{\F}_p, \text{ss}} \simeq (\textbf{F} \otimes_{\bar{\F}_p} \hat{H}^1_{\text{LT}, \bar{\F}_p})^{\bar{D}^\times (F_\mfrak{p})}. \label{lt}
\end{align}
 We would like to compare the mod $p$ Hecke algebras of $D^\times$ and $\bar{D}^\times$. Let us denote ${\bf{T}}^{D}(K_\Sigma)$ and ${\bf{T}}^{\bar{D}}(K_\Sigma)$ to be the Hecke algebra corresponding to $D$ and $\bar{D}$ respectively and $T_w$ for the Hecke operators for $w\not\in \Sigma$ (see \cite[Section 6.3]{Cho15} for details). Taking $K^\Sigma$-invariant on $\bf{F}$, there is an action of the Hecke algebra ${\bf{T}}^{\bar{D}}(K_\Sigma)$. Following we state a Serre type conjecture comparing the eigenvalues of the $\bar{\F}_p$-Hecke algebras which is a generalisation of \cite[TH\'EOR\`EME, p.282]{Serre96}.
\begin{conj}{\label{Serretype}}
Let $D$ and $\bar{D}$ be as above. Then the systems of eigenvalues for $(T_w)$ of ${\bf{T}}^{\bar{D}}_{\Sigma}$ on $\bf{F}$ are in bijection with the systems of eigenvalues for $(T_w)$ of ${\bf{T}}^D(K_\Sigma)$ coming from mod $p$ modular forms. 
\end{conj}
\subsection*{Proof of Theorem \ref{LT}}
\begin{proof}
The above conjecture allows us to identify the maximal ideals of ${\bf{T}}^D(K_\Sigma)$ with those of ${\bf{T}}^{\bar{D}}(K_\Sigma)$, so let $\mfrak{m}_\rho$ be the maximal ideal corresponding to $\rho$ as in Section~\ref{new vectors} and define $\sigma_\mfrak{m_\rho} := {\bf{F}}[\mfrak{m_\rho}]^K$ (where $K= K_{\Sigma'}K^\Sigma$ be as in Section \ref{the proof}) which is a representation of $\bar{D}^\times (F_{\mfrak{p}})$. Taking $K$-invariance in \eqref{lt} which commutes with $\bar{D}^\times (F_\mfrak{p})$-invariants, we get
\begin{align*}
(H^1_{\bar{\F}_p, \text{ss}})^K \simeq (\textbf{F}^K \otimes_{\bar{\F}_p} \hat{H}^1_{\text{LT}, \bar{\F}_p})^{\bar{D}^\times (F_\mfrak{p})}.
\end{align*}
Following \cite[Section 6.3]{Cho15}, we define $\sigma_{\mfrak{m}_\rho}^\vee = \Hom_{\bar{\F}_p}(\sigma_{\mfrak{m}_\rho}, \bar{\F}_p)$. Taking $[\mfrak{m}_\rho]$-part we get
\begin{align*}
(H^1_{\bar{\F}_p, \text{ss}}[\mfrak{m}_\rho])^K \simeq (\sigma_{\mfrak{m}_\rho} \otimes_{\bar{\F}_p} \hat{H}^1_{\text{LT}, \bar{\F}_p})^{\bar{D}^\times (F_\mfrak{p})}.
\end{align*}
Let $\hat{H}^1_{\text{LT}, \bar{\F}_p}[\sigma_{\mfrak{m}_\rho}^\vee]:=(\sigma_{\mfrak{m}_\rho} \otimes_{\bar{\F}_p} \hat{H}^1_{\text{LT}, \bar{\F}_p})^{\bar{D}^\times (F_\mfrak{p})}$, then by Theorem \ref{Mainthm}, we have the following injection
\begin{align*}
\rho\otimes \pi_\mfrak{p}\otimes (\otimes_{w\in \Sigma\setminus \Sigma'}\pi_w) \hookrightarrow  \emph{H}^1_{\bar{\F}_p, ss}[\mfrak{m}_\rho]^K \simeq \hat{H}^1_{\text{LT}, \bar{\F}_p}[\sigma_{\mfrak{m}_\rho}^\vee], 
\end{align*}
proving Theorem \ref{LT}.
\end{proof}

\section{Bounds on the cohomology of modular curve supported on ordinary locus}

Let $\bar{\rho}: G_{\Q} :=\rm{Gal}(\bar{\Q}/\Q) \rightarrow \GL_2(\bar{\F}_p)$ be a continuous, absolutely irreducible Galois representation (hence modular) and let $\mfrak{m}$ be the maximal ideal of the corresponding Hecke algebra associated to $\bar{\rho}$.
 We define $\Sigma_0$ to be the set of primes dividing the Artin conductor of $\bar{\rho}$ (\cite[Section 1.2]{Serre87}). Let $K_{\Sigma_0} \subset \Pi_{l\in\Sigma_0}\GL_2(\Z_l)$ be an allowable level (see \cite[Section 6]{Morra13}), $I_t$ be the Iwahori subgroup of $\GL_2(\Z_p)$ and let $K_0^\Sigma = \Pi_{l\not\in \Sigma_0} \GL_2(\Z_l)$. Fixing the level $I_tK_{\Sigma_0}K_0^\Sigma$ and taking the $[\mfrak{m}]$-torsion in the above long exact sequence we get 
$$\cdots\rightarrow H^0(X(I_tK_{\Sigma_0}K_0^\Sigma)_{\rm{ss}}, \bar{\F}_p)[\mfrak{m}]\xrightarrow{f} H^1_{X_{\rm{ord}}}(X(I_tK_{\Sigma_0}K_0^\Sigma)^{an}, \bar{\F}_p)[\mfrak{m}]\xrightarrow{g} H^1(\cX(I_tK_{\Sigma_0}K_0^\Sigma)^{an}, \bar{\F}_p)[\mfrak{m}] \cdots$$
Let $d:=\dim_{\bar{\F}_p}(\bigotimes_{l\in\Sigma_0}\pi(\bar{\rho}|_{G_{\Q_l}}))^{K_{\Sigma_0}}$, where $\pi(\bar{\rho}|_{G_{\Q_l}})$ is the smooth mod $p$ representation of $\GL_2(\Q_l)$ associated to $\bar{\rho}|_{G_{\Q_l}}$ by Emerton-Helm and let 
\[
s= \text{ number of supersingular points on the modular curve } X(I_tK_{\Sigma_0}K_0^\Sigma).
\] 
By \cite[Theorem 1.2]{MR3650223}, we can compute $h_{\overline{\rho}_p}=\dim H^1_{X_{\rm{ord}}}(\cX_0(Np^m)^{an}, \bar{\F}_p)[\mfrak{m}]$.  Note that this quantity depends on the local Galois representation $\overline{\rho}_p$. Let $\Sigma_0$ be the set of primes diving the Artin conductor of $\rho$, let $K_{\Sigma_0}$ be an allowable level for $\bar{\rho}$,
\begin{lemma}
 Let $p\geq 3$ and $\bar{\rho}$ be as in  \cite[Theorem 1.2]{MR3650223}. Then 
 \[
 \dim H^1_{X_{\rm{ord}}}(\cX_0(Np^m)^{an}, \bar{\F}_p)[\mfrak{m}] \leq h_{\overline{\rho}_p} +s.
 \]
\end{lemma}
\begin{proof}
By the above exact sequence we have $$\dim H^1_{X_{\rm{ord}}}(X_0(Np^m)^{an}, \bar{\F}_p)[\mfrak{m}] = \dim(\text{image of }f) + \dim (\text{image of g}).$$ 
By  we have that $\dim (\text{image of g})\leq h_{\overline{\rho}_p}$ and $\dim (\text{image of f}) \leq $ number of supersingular points on the modular curve $ X(I_tK_{\Sigma_0}K_0^\Sigma)$. Therefore 
\[
\dim H^1_{X_{\rm{ord}}}(X_0(Np^m)^{an}, \bar{\F}_p)[\mfrak{m}] \leq h_{\overline{\rho}_p} + s.
\]
\end{proof}

\bibliographystyle{alpha}
\bibliography{domp}
\end{document}